\documentclass[12pt,a4paper]{amsart}

\usepackage{amsmath,amsthm,amssymb,amsfonts}
\usepackage{graphicx}
\usepackage{color}
\usepackage[active]{srcltx}

\newtheorem{theorem}{Theorem}
\newtheorem{cor}{Corollary}
\newtheorem{prop}{Proposition}

\newtheorem{defs}{Definition}
\newtheorem{example}{Example}

\newtheorem{rem}{Remark}
\newcommand{\ds}{\displaystyle}
\newcommand{\R}{\mathbb{R}}
\newcommand{\N}{\mathbb{N}}

\newcommand{\C}{\mathbb{C}}
\newcommand{\K}{\mathbb{K}}

\newcommand{\e}{\varepsilon}

\DeclareMathOperator{\supp}{supp}
\DeclareMathOperator{\dist}{dist}

\DeclareMathOperator{\re}{Re}

\numberwithin{equation}{section}

\title[The strong Bishop-Phelps-Bollob\'as Property]
{The strong Bishop-Phelps-Bollob\'as Property}

\author[S. Dantas]{Sheldon Dantas}
\address{Departamento de An\'{a}lisis Matem\'{a}tico,
Universidad de Valencia, Doctor Moliner 50, 46100 Burjasot
(Valencia), Spain} \email{sheldon.dantas@uv.es}

\thanks{
The author was supported by MINECO MTM2014-57838-C2-2-P and CAPES, Doutorado Pleno CSF, BEX 0050/13-0}


\date{}


\begin{document}

\begin{abstract}
In this paper we introduce the strong Bishop-Phelps-Bollob\'as property (sBPBp) for bounded linear operators between two Banach spaces $X$ and $Y$. This property is motivated by a Kim-Lee result which states, under our notation,  that a Banach space $X$  is uniformly convex if and only if the pair $(X,\K)$ satisfies the sBPBp. Positive results of pairs of Banach spaces $(X,Y)$  satisfying this property are given  and  concrete  pairs of Banach spaces $(X, Y)$  failing it are exhibited. A complete characterization of the sBPBp for the pairs $(\ell_p, \ell_q)$ is also provided.
 \end{abstract}

\maketitle

\section{Introduction}
\par

Let $X$ and $Y$ be Banach spaces over a real or complex field $\K$. We use the traditional notations $S_X$ and $B_X$ for the unit sphere and the closed unit ball of the space $X$, respectively. The Banach space of all bounded linear operators $T: X \to Y$ will be represented by $\mathcal{L}(X, Y)$. In particular, when $Y = \K$ we denote $\mathcal{L}(X, \K)$ simply by putting $X^*$ called the dual space of $X$. We say that an operator $T \in \mathcal{L}(X, Y)$ attains its norm if there exists $x_0 \in S_X$ such that $\|T(x_0)\| = \|T\| = \sup_{x \in S_X} \|T(x)\|$. In this case, we say that $T$ is a norm attaining operator and it attains its norm at $x_0$. The subset of $\mathcal{L}(X, Y)$ of all norm attaining operators is denoted by $NA(X, Y)$. We recall that a bounded linear operator is compact if the closure of the image of the unit ball is compact. For $1 \leq p \leq \infty$ we denote by $\ell_p^n (\K)$ the euclidean space $\K^n$ endowed with the $p$-norm $\|x\|_p^p := |x_1|^p + \ldots + |x_n|^p$ with $x = (x_1, \ldots, x_n) \in \K^n$ and $1 \leq p < \infty$ and $\ell_{\infty}^n$ endowed with the sup-norm $\|x\|_{\infty} := \sup_{j \in \N} |x_j|$. To simplify the notation we put just $\ell_p^n$ when the field that we are working is specified.

The Bishop-Phelps theorem \cite{BP} says that every bounded linear functional can be approximated by norm attaining functionals. In other words, this means that the set of all norm attaining functionals on a Banach space $X$ is dense in its dual space $X^*$. It was proved by Lindenstrauss \cite{Lind} in 1963 that, in general, the same result does not work for bounded linear operators. More precisely, he exhibited a Banach space $X$ such that the set $NA(X, X)$ is not dense in $\mathcal{L}(X, X)$. Seven years later, Bollob\'as proved a numerical version of the Bishop-Phelps theorem which nowdays is known as the Bishop-Phelps-Bollob\'as theorem \cite{Bol}. As a consequence of \cite[Theorem 2.1]{CKMM} we may enunciate this theorem as follows.

\begin{theorem} \emph{(Bishop-Phelps-Bollob\'as theorem, \cite{BP}, \cite{CKMM})} Let $\e \in (0, 2)$ and suppose that $x_0 \in B_X$ and $x_0^* \in B_{X^*}$ satisfy
\begin{equation*}
\re x_0^*(x_0) > 1 - \frac{\e^2}{2}.
\end{equation*}
Then, there are $x_1 \in S_X$ and $x_1^* \in S_{X^*}$ such that
\begin{equation*}
|x_1^*(x_1)| = 1, \ \ \|x_1 - x_0\| < \e \ \ \mbox{and} \ \ \|x_1^* - x_0^*\| < \e.
\end{equation*}
\end{theorem}	
Since we do not have a Bishop-Phelps version for bounded linear operators, we can not expect a Bishop-Phelps-Bollob\'as version for this type of functions either. So it is natural to study the conditions that the Banach spaces $X$ and $Y$ must satisfy to get a theorem of this nature. In 2008, Acosta, Aron, Garc\'ia and Maestre introduced a definition in order to attain this problem.

\begin{defs} \label{BPBp} (Bishop-Phelps-Bollob\'as property, \cite{AAGM2}) We say that a pair of Banach spaces $(X, Y)$ satisfies the {\it Bishop-Phelps-Bollob\'as property} (BPBp, for short) when given $\e > 0$, there exists $\eta(\e) > 0$ such that whenever $T \in S_{\mathcal{L}(X, Y)}$ and $x_0 \in S_X$ are such that
\begin{equation*}
\|T(x_0)\| > 1 - \eta(\e),
\end{equation*}
there are $S \in S_{\mathcal{L}(X, Y)}$ and $x_1 \in S_X$ such that
\begin{equation*}
\|S(x_1)\| = 1, \ \ \ \|x_1 - x_0\| < \e \ \ \ \mbox{and} \ \ \ \|S - T\| < \e.
\end{equation*}
\end{defs}

There are many classical Banach spaces that satisfy the BPBp. For example, when $X$ and $Y$ are finite dimensional spaces, the pair $(X, Y)$ has this property \cite[Proposition 2.4]{AAGM2}. Also, if $Y$ has the property $\beta$ of Lindenstrauss, as $c_0$ and $\ell_{\infty}$ do, then the pair $(X, Y)$ satisfies the BPBp for all Banach space $X$ \cite[Theorem 2.2]{AAGM2}. More positive results appear when we assume that the range space $Y$ is uniformly convex: the pairs $(\ell_{\infty}^n, Y)$, $(c_0, Y)$ and $(L_{\infty}(\mu), Y)$ all satisfy the BPBp (see \cite{AAGM2}, \cite{Kim} and \cite{KLL}, respectively). Also if $X$ is a uniformly convex Banach space, then the pair $(X, Y)$ has the BPBp for all Banach space $Y$ \cite[Theorem 3.1]{KL}.

Just to help make the paper entirely accessible, we remember the concept of uniform convexity. A Banach space $X$ is uniformly convex if given $\e > 0$, there exists $\delta(\e) > 0$ such that whenever $x_1, x_2 \in S_X$ satisfy $\| x_1 - x_2\| \geq \e$, then $\|\frac{1}{2}(x_1 + x_2) \| \leq 1 - \delta(\e)$. We recall that if $p \in (1, \infty)$ then $\ell_p$ is uniformly convex.

In 2014, Kim and Lee \cite[Theorem 2.1]{KL} given a characterization for uniformly convex Banach spaces that associate this type of spaces with a peculiarity on the Bishop-Phelps-Bollob\'as property. More precisely, they proved that a Banach space $X$ is uniformly convex if and only if given $\e > 0$ then we are able to find a positive number $\eta(\e) > 0$ such that whenever $x^* \in S_{X^*}$ and $x_0 \in S_X$ satisfy the relation
\begin{equation*}
|x^*(x_0)| > 1 - \eta(\e),
\end{equation*}
there exists a vector $x_1 \in S_X$ such that
\begin{equation*}
|x^*(x_1)| = 1 \ \ \ \mbox{and} \ \ \ \|x_0 - x_1\| < \e.
\end{equation*}
Note that the theorem says that a Banach space $X$ is uniformly convex if and only if the pair $(X, \K)$ satisfies the Bishop-Phelps-Bollob\'as property without changing the initial functional $x^*$, that is, the functional that almost attains its norm at some point $x_0$ is the same that attains its norm at the new vector that is close to $x_0$.

In this paper, we study this last result for bounded linear operators and we call it as the strong Bishop-Phels-Bollob\'as property (sBPBp). First, we study it in the case that the real number $\eta( \ . \ )$ depends of $\e > 0$ and also of a fixed operator $T$, and we get some positives results about it. After that, we study the property in the uniform case, that is, when the number $\eta$ depends only of $\e > 0$ as we are used to work when we work with the BPBp. As we will see in the next section, we get many negatives results about the uniform case and we use these results to prove that there are uniformly convex Banach spaces $X$ and infinite dimensional Banach spaces $Y$ such that the pair $(X, Y)$ fails the sBPBp. Finally, we give a complete characterization to the sBPBp for the pair $(\ell_p, \ell_q)$ which describes when these pairs satisfy the property.

\section{The Strong Bishop-Phelps-Bollob\'as Property}

In this section we study the strong Bishop-Phelps-Bollob\'as property. Namely, we study the conditions that the Banach spaces $X$ and $Y$ must have and the hypothesis that we have to add to get a Kim-Lee type theorem for bounded linear operators. The Kim-Lee theorem is enunciated as follows.

\begin{theorem} \label{Kim-Lee}
\emph{(Kim-Lee theorem, \cite{KL})} A Banach space $X$ is uniformly convex if and only if given $\e > 0$, there exists $\eta(\e) > 0$ such that whenever $x^* \in S_{X^*}$ and $x_0 \in S_X$ satisfy
\begin{equation*}
|x^*(x_0)| > 1 - \eta(\e),
\end{equation*}
there is $x_1 \in S_X$ such that
\begin{equation*}
|x^*(x_1)| = 1 \ \ \ \mbox{and} \ \ \ \|x_1 - x_0\| < \e.
\end{equation*}
\end{theorem}
As we mentioned before, it is like a Bishop-Phelps-Bollob\'as property without changing the initial functional $x^*$. A natural question arises: the result is true for other pairs of Banach spaces $(X, Y)$ with $X$ and $Y$ having additional hypothesis considering bounded linear operators instead bounded linear functionals? Although it is more natural put the question just like that, it is seems to us to be a strong problem in the sense that it will be hard to find concrete pairs of Banach spaces $(X, Y)$ satisfying the Kim-Lee theorem for bounded linear operators. So we start by considering that the positive real number $\eta( \ . \ ) > 0$ that appears in Theorem \ref{Kim-Lee} does not depend only of $\e > 0$ but also of a given operator $T$ fixed. Before we do that, we want to comment that Carando, Lassalle and Mazzitelli \cite{Carando} defined a weak BPBp for ideals of multilinear mappings where the real positive number $\eta ( \ . \ )$ in the definition of the BPBp in this context depends of a given $\e > 0$ and also of the ideal norm of the operator defined on a normed ideal of $N$-linear mappings. In other words, a normed ideal of $N$-linear mappings $\mathcal{U} = \mathcal{U}(X_1 \times \ldots \times X_N; Y)$ where $X_1, \ldots, X_N, Y$ are Banach spaces has the weak BPBp if for each $\varPhi \in \mathcal{U}$ with $\|\varPhi\| = 1$ and $\e > 0$, there exists $\eta(\e, \|\varPhi\|_{\mathcal{U}}) > 0$ depending also of $\| \varPhi \|_{\mathcal{U}}$ such that if $(x_1, \ldots, x_N) \in S_{X_1} \times \ldots \times S_{X_N}$ satisfies $\| \varPhi (x_1, \ldots, x_N)\| > 1 - \eta(\e, \|\varPhi\|_{\mathcal{U}})$, then there exist $\varPsi \in \mathcal{U}$ with $\|\varPsi\| = 1$ and $(a_1, \ldots, a_N) \in S_{X_1} \times \ldots \times S_{X_N}$ such that $\| \varPsi(a_1, \ldots, a_N)\| = 1$, $\| (a_1, \ldots, a_N) - (x_1, \ldots, x_N)\| < \e$ and $\| \varPsi - \varPhi\|_{\mathcal{U}} < \e$. They proved that if $X_1, \ldots, X_N$ are uniformly convex Banach spaces then $\mathcal{U}$ has the weak BPBp for ideals of multilinear mappings for all Banach space $Y$. Here we will work on a different context where $\eta( \ . \ )$ depends of a fixed operator not of the norm of the operators ideal as we may see in the following definition.

\begin{defs} \label{sBPBp} A pair of Banach spaces $(X, Y)$ has the \emph{strong Bishop-Phelps-Bollob\'as property} (or sBPBp for short) if given $\e > 0$ and $T \in S_{\mathcal{L}(X, Y)}$, there exists $\eta(\e, T) > 0$ such that whenever $x_0 \in S_X$ satisfies
\begin{equation*}
\|T(x_0)\| > 1 - \eta(\e, T),
\end{equation*}
there is $x_1 \in S_X$ such that
\begin{equation*}
\|T(x_1)\| = 1 \ \ \ \mbox{and} \ \ \ \|x_1 - x_0\| < \e.
\end{equation*}
If $T \in \mathcal{L}(X, Y)$ is compact and the pair $(X, Y)$ has the sBPBp, then we say that the pair $(X, Y)$ has the sBPBp for compact operators.
\end{defs}

In the next proposition, we assume that the domain space $X$ is finite dimensional to get the sBPBp first result. In fact, as we mentianed in the previous paragraph, we get a positive real number $\eta(\e, T) > 0$ that depends of $\e > 0$ and also of the operator $T \in S_{\mathcal{L}(X, Y)}$ fixed at the beggining of the proof.

\begin{theorem} Let $X$ be a finite dimensional Banach space. Then the pair $(X, Y)$ has the sBPBp for all Banach space $Y$.
\end{theorem}

\begin{proof} The proof is by contradiction. Let $T \in S_{\mathcal{L}(X, Y)}$. If the result is false for some $\e_0 > 0$, then for all $n \in \N$, there exists $x_n \in S_X$ such that
\begin{equation*}
\|T(x_n)\| > 1 - \frac{1}{n}
\end{equation*}
but $\dist (x_n, NA(T)) \geq \e_0$, where $NA(T) = \{ z \in S_X: \|T(z)\| = 1 \}$ for $n \in \N$. Since $X$ is finite dimensional, there exists a subsequence $(x_{n_k})$ of $(x_n)$ such that $x_{n_k} \longrightarrow x_0$ for some $x_0 \in X$. This implies that $\|T(x_{n_k})\| \longrightarrow \|T(x_0)\|$ and since
\begin{equation*}
1 \geq \|T(x_{n_k})\| \geq 1 - \frac{1}{n}
\end{equation*}
we get that $\|T(x_0)\| = \|x_0\| = 1$ and so $x_0 \in NA(T)$. Then $NA(T) \not= \emptyset$ and
\begin{equation*}
\e \leq \dist(x_{n_k}, NA(T)) \leq \|x_{n_k} - x_0\| \stackrel{k \to \infty}{\longrightarrow}        0
\end{equation*}
which is a contradiction.
\end{proof}

If we assume that $X$ is uniformly convex and the fixed bounded linear operator $T: X \longrightarrow Y$ is compact, we get that the pair $(X, Y)$ has the sBPBp as we may see in the next theorem.

\begin{theorem} \label{sBPBp2} Let $X$ be a uniformly convex Banach space. Then the pair $(X, Y)$ has the sBPBp for compact operators for all Banach space $Y$.
\end{theorem}

\begin{proof} The proof is again by contradiction. Let $T \in S_{\mathcal{L}(X, Y)}$ be a compact operator. If the result is false, for some $\e_0 > 0$ and for all $n \in \N$, there exists $x_n \in S_X$ such that
	\begin{equation*}
	1 \geq \|T(x_n)\| > 1 - \frac{1}{n}
	\end{equation*}
but $\dist (x_n, NA(T)) \geq \e_0$ for all $n \in \N$. Then $\|T(x_n)\| \longrightarrow 1$ as $n \to \infty$. Since $X$ is uniformly convex, $X$ reflexive and then by the Smulian theorem there exists a subsequence $(x_{n_k})$ of $(x_n)$ and $x_0 \in X$ such that $(x_{n_k})$ converges weakly to $x_0$. Since $T$ is completely continuous, $T(x_{n_k})$ converges in norm to $T(x_0)$ as $k \longrightarrow \infty$. Therefore $\|T(x_0)\| =  \|x_0\| = 1$ and so $x_0 \in NA(T)$. Thus $NA(T) \not= \emptyset$ and
\begin{equation*}
1 \geq \left\| \frac{x_n + x_0}{2} \right\| \geq \left\| \frac{T(x_n) + T(x_0)}{2} \right\| \stackrel{k \to \infty}{\longrightarrow} \|T(x_0)\| = 1.
\end{equation*}
This implies that $\lim_{k \to \infty} \|x_{n_k} + x_0\| = 2$ and, using again that $X$ is uniformly convex, we get that $$\lim_{k \to \infty} \|x_{n_k} - x_0\| = 0 ,$$ which is a contradiction because of the following inequalities:
\begin{equation*}
\e_0 \leq \dist (x_{n_k}, NA(T)) \leq \|x_n - x_0\| \stackrel{k \to \infty}{\longrightarrow} 0.
\end{equation*}

\end{proof}

Consequently we get two more positives results about the sBPBp. In the next corollary we prove that the pair $(X, Y)$ has the sBPBp whenever $X$ is uniformly convex and $Y$ has the Schur's property which implies as a particular case that the pair $(\ell_2, \ell_1)$ satisfies the property. Corollary \ref{sBPBp3} shows that whenever $X$ is uniformly convex and $Y$ has finite dimension, the pair $(X, Y)$ has the sBPBp by using the fact that every bounded linear operator with finite dimensional range is compact.

\begin{cor} \label{sBPBp6} If $X$ is a uniformly convex Banach space and $Y$ is a Banach space with the Schur's property, then the pair $(X, Y)$ has the sBPBp. In particular, $(\ell_2, \ell_1)$ has the sBPBp.
\end{cor}

\begin{proof} We apply Theorem \ref{sBPBp2}. To do this, we prove that every bounded linear operator $T: X \longrightarrow Y$
is compact. Indeed, since $T$ is continuous, $T$ is $w$-$w$ continuous. Let $(x_n)_n \subset B_X$. Since $X$ is reflexive, by the Smulian theorem, there are a subsequence of $(x_n)_n$ (which we denote again by $(x_n)_n$) and $x_0 \in X$ such that $x_n \stackrel{w}{\longrightarrow} x_0$. So $T(x_n) \stackrel{w}{\longrightarrow} T(x_0)$. Now, since $Y$ has the Schur's property, $T(x_n) \longrightarrow T(x_0)$ in norm. So $T$ is compact. By Theorem \ref{sBPBp2} the pair $(X, Y)$ has the sBPBp.
\end{proof}

\begin{cor} \label{sBPBp3} If $X$ is a uniformly convex Banach space and $Y$ is a finite dimensional Banach space, then the pair $(X, Y)$ has the sBPBp.
\end{cor}

\begin{rem} \label{sBPBp5} Note that in Definition \ref{sBPBp} the operator $T: X \longrightarrow Y$ must attains its norm if the pair $(X, Y)$ has the sBPBp. So if $X$ is not reflexive, then the pair $(X, Y)$ fails the sBPBp for all Banach space $Y$. Indeed, since $X$ is not reflexive, by the James theorem, there is a linear continuous functional $x_0^* \in S_{X^*}$ such that $|x_0^*(x)| < 1$ for all $x \in S_X$. Let $y_0 \in S_Y$ and define $T: X \longrightarrow Y$ by $T(x) := x_0^*(x) y_0$. Then $\|T\| = \|x_0^*\| = 1$ and $\|T(x)\| = |x_0^*(x)| < 1$ for all $x \in S_X$. This implies that $T$ never attains its norm and then the pair $(X, Y)$ can not have the sBPBp.
\end{rem}

Note that in the classic definition of the Bishop-Phelps-Bollob\'as property the number $\eta( \ . \ )$ depends only of $\e > 0$. So what happen if we ask for more in the definition of the strong Bishop-Phelps-Bollob\'as property? As we will see below, when we put $\eta(\ . \ )$ to depends only of $\e > 0$ we get negative results. However, we use them to get examples of pairs of Banach spaces that fail the sBPBp when the domain space is reflexive or when the range space has infinite dimension. Just to help to make reference we put a name of it.

\begin{defs} We say that a pair of Banach spaces $(X, Y)$ has the \emph{uniform strong Bishop-Phelps-Bollob\'as property} (uniform sBPBp, for short) if given $\e > 0$, there exists $\eta(\e) > 0$ such that whenever $T \in S_{\mathcal{L}(X, Y)}$ and $x_0 \in S_X$ are such that
\begin{equation*}
\|T(x_0)\| > 1 - \eta(\e),
\end{equation*}
there is $x_1 \in S_X$ such that
\begin{equation*}
\|T(x_1)\| = 1 \ \ \ \mbox{and} \ \ \ \|x_1 - x_0\| < \e.
\end{equation*}
\end{defs}

We observe that the Kim-Lee theorem says that a Banach space $X$ is uniformly convex if and only if the pair $(X, \K)$ has the uniform sBPBp where $\K = \R$ or $\C$. Note also that if the pair $(X, Y)$ satisfies the uniform sBPBp then the pair $(X, Y)$ satisfies the BPBp. The first thing that we notice is that if $(X, Y)$ has the uniform sBPBp for some Banach space $Y$, then $X$ must be uniformly convex.

\begin{prop} Let $X$ be a Banach space. If there exists a Banach space $Y$ such that the pair $(X, Y)$ has the uniform sBPBp, then the pair $(X, \K)$ has the uniform sBPBp.	
\end{prop}

\begin{proof} Given $\e \in (0, 1)$, consider $\eta(\e) > 0$ the positive real number that satisfies the uniform sBPBp for the pair $(X, Y)$. We prove that the pair $(X, \K)$ has the uniform sBPBp with the same $\eta(\e)$. Indeed, let $x^* \in S_{X^*}$ and $x_0 \in S_X$ be such that
\begin{equation*}
|x^*(x_0)| > 1 - \eta(\e).
\end{equation*}
Let $y_0 \in S_Y$ and define $T \in \mathcal{L}(X, Y)$ by $T(x) := x^*(x) y_0$ for all $x \in X$. So $\|T\| = \|x^*\| = 1$ and
\begin{equation*}
\|T(x_0)\| = |x^*(x_0)| > 1 - \eta(\e).
\end{equation*}
Since the pair $(X, Y)$ has the uniform sBPBp with $\eta(\e)$, there exists $x_1 \in S_X$ such that $\|T(x_1)\| = 1$ and $\|x_0 - x_1\| < \e$. Since $\|T(x_1)\| = |x^*(x_1)|$ the proof is complete.
\end{proof}

By the Kim-Lee theorem, we have the following consequence.

\begin{cor} \label{CorsBPBp} Let $X$ be a Banach space. If there exists a Banach space $Y$ such that the pair $(X, Y)$ has the uniform sBPBp, then $X$ is uniformly convex.
\end{cor}

By this corollary, since $\ell_1^2$ is not uniformly convex, all the pairs $(\ell_1^2, Y)$ fail the uniform sBPBp for any Banach space $Y$.

What about the reciprocal of Corollary \ref{CorsBPBp}? The first thing that come to mind, since every Hilbert space is uniformly convex, is to assume that the domain space $X$ is a Hilbert space and try to find some Banach space $Y$ such that the pair $(X, Y)$ satisfies the property. But even in the simplest situation the result fails as we may see in the following example.

\begin{example} \label{exsBPBp} This example works for both real and complex cases. For a given $\e > 0$, suppose that there exists $\eta(\e) > 0$ satisfying the uniform sBPBp for the pair $(\ell_2^2, \ell_{\infty}^2)$.  Let $T: \ell_2^2 \to \ell_{\infty}^2$ be defined by
\begin{equation*}
T(x, y) := \left( \left(1 - \frac{1}{2} \eta(\e)\right)x, y \right)
\end{equation*}
for every $(x, y) \in \ell_2^2$. For every $(x, y) \in S_{\ell_2^2}$, we have
\begin{equation*}
\|T(x, y)\|_{\infty} = \left\| \left( \left(1 - \frac{1}{2} \eta(\e) \right) x, y \right) \right\|_{\infty} \leq 1.
\end{equation*}
Since $T(e_2) = 1$, we obtain $\|T\| = 1$. Moreover,
\begin{equation*}
\|T(e_1)\| = 1 - \frac{1}{2} \eta(\e) > 1 - \eta(\e).
\end{equation*}
We prove now that every $z = (a, b) \in S_{\ell_2^2}$ such that $\|T(z)\|_{\infty} = 1$ assumes the form $z = \lambda e_2$ for $|\lambda| = 1$. Indeed, since
$\left| 1 - \frac{1}{2} \eta(\e) \right| < 1$ and $\|T(z)\|_{\infty} = 1$, we have $|b| = 1$. Since $|a|^2 + |b|^2 = 1$, we have $a = 0$ and $b = \lambda$ with $|\lambda| = 1$. In summary, we have a unit operator $T$ and a unit vector $e_1$ satisfying $\|T(e_1)\| > 1 - \eta(\e)$ but if $T$ attains its norm at some point $z \in S_{\ell_2^2}$ then $z = (0, \lambda)$ with $|\lambda| = 1$. This contradicts the assumption that the pair $(\ell_2^2, \ell_{\infty}^2)$ has the uniform sBPBp since $z$ is far from $e_1$ in view of the fact that $\|e_1 - z\|_2 = \|(1, \lambda)\|_2 = \sqrt{2}$.
\end{example}

This shows that the pair $(\ell_2^2 (\K), \ell_{\infty}^2(\K))$ fails to have the uniform sBPBp for $\K = \R$ or $\C$. Now what if we add on the hypothesis that both $X$ and $Y$ are Hilbert spaces? The answer for this question is still no as we can see below.

\begin{example}  \label{sBPBp8} This example works for both real and complex cases. For a given $\e > 0$, suppose that there exists $\eta(\e) > 0$ satisfying the uniform sBPBp for the pair $(\ell_2^2, \ell_2^2)$. Define $T: \ell_2^2 \to \ell_2^2$ by
	\begin{equation*}
	T(x, y) := \left( \left(1 - \frac{\eta(\e)}{2} \right)x, y \right),
	\end{equation*}
	for every $(x, y) \in \ell_2^2$. Then for every $(x, y) \in S_{\ell_2^2}$,
	\begin{equation*}
	\|T(x, y)\|_2^2 =  \left( 1 - \frac{\eta(\e)}{2} \right)^2 |x|^2 + |y|^2  \leq |x|^2 + |y|^2 = \|(x, y)\|_2 = 1
	\end{equation*}
	and since $\|T(e_2)\|_2 = \|e_2\|_2 = 1$, we have $\|T\| = 1$. Also, we see that $\|T(e_1)\|_2 = 1 - \frac{\eta(\e)}{2} > 1 - \eta(\e)$. Now, if $z = (a, b) \in S_{\ell_2^2}$ is such that $\|T(z)\|_2 = 1$, then
	\begin{equation*}
	\left( 1 - \frac{\eta(\e)}{2} \right)^2 |a|^2 + |b|^2 = 1
	\end{equation*}
	and since $|b|^2 = 1 - |a|^2$, we get
	\begin{equation*}
	\eta (\e) \left( \frac{\eta(\e)}{4} - 1 \right) |a|^2 = 0
	\end{equation*}
	which implies that $a = 0$ and using again that $|a|^2 + |b|^2 = 1$, we obtain $b = \lambda e_2$ with $|\lambda| = 1$ and so $\|e_1 - z\|_2 = \sqrt{2}$ which contradicts the hypothesis that the pair $(\ell_2^2, \ell_2^2)$ has the uniform sBPBp.
\end{example}

More in general, we have the following positive result.

\begin{prop}\label{ParaMiguel1} Let $\ell_p^2$ and $\ell_q^2$ be the space $\K^2$ endowed with the $p$-norm and the $q$-norm, respectively, with $1 < p \leq q < \infty$ (or $p < q = \infty$). Given $\beta \in (0, 1)$, there exists a bounded linear operator $T_{\beta}: \ell_p^2 \rightarrow \ell_q^2$ with $\|T_{\beta}\| = 1$ such that
	\begin{itemize}
		\item[(i)] $\|T_{\beta} (e_1)\|_q = \beta$ and
		\item[(ii)] for every $z \in S_{\ell_p^2}$ such that $\|T_{\beta} (z)\|_q = 1$, we have $\|z - e_1\|_p = 2^{\frac{1}{p}}.$
	\end{itemize}
\end{prop}

\begin{proof} Let $\beta \in (0, 1)$ and $1 < p \leq q < \infty$. Define $T_{\beta}: \ell_p^2 \rightarrow \ell_q^2$ by $T_{\beta} (x, y) := (\beta x, y)$ 	for every $(x, y) \in \ell_p^2$. If $\|(x, y)\|_p = 1$, since $p \leq q$, we get
\begin{equation*}
\|T_{\beta} (x, y)\|_q = (\beta^q |x|^q + |y|^q)^{\frac{1}{q}} < (|x|^p + |y|^p)^{\frac{1}{q}} = 1,
\end{equation*}
which implies that $\|T_{\beta}\| \leq 1$. Since $\|T_{\beta} (e_2)\|_q = \|e_2\|_q = 1$, we have $\|T_{\beta}\| = 1$. Now, let $z = (a, b)\in S_{\ell_p^2}$ be such that $\|T_{\beta}(z)\|_q = 1$. We prove that $b = \lambda e_2$ with $|\lambda| = 1$. Indeed, the equality $\|T_{\beta}(a, b)\|_q = 1$ implies that $\beta^q |a|^q + |b|^q = 1$ and since $|a|^p + |b|^p = 1$, we do the difference between these two equalities to get
	\begin{equation*}
	(|a|^p - \beta^q |a|^q) + (|b|^p - |b|^q) = 0.
	\end{equation*}
	Since $p \leq q$ and $|a|, |b| \leq 1$, $|a|^p - \beta^q|a|^q \geq 0$ and $|b|^p - |b|^q \geq 0$. Because of the above equality, we get that $|a|^p - \beta^q |a|^q = 0 = |b|^p - |b|^q$. But $|a|^q \leq |a|^p$ which implies that
	\begin{equation*}
	0 = |a|^p - \beta^q |a|^q \geq (1 - \beta^q)|a|^p.
	\end{equation*}
	Thus $a = 0$ and then $b = \lambda e_2$ with $|\lambda| = 1$ as desired. So if $z \in S_{\ell_p^2}$ is such that $\|T_{\beta} (z)\|_q = 1$, then $\|z - e_1\|_p = 2^{\frac{1}{p}}$ which completes the proof.
\end{proof}

As a consequence of this last result, we get that  all the pairs $(\ell_p (\K), \ell_q (\K))$ fail the uniform sBPBp for $1 < p \leq q < \infty$ for $\K = \R$ or $\C$. Now we study the pair $(\ell_2^2, \ell_1^2)$ in the real case when we put the sum norm on the range space. Unfortunatly, we can construct a bounded linear operator to get the same contradiction as before.

\begin{prop} Let $\ell_2^2$ and $\ell_1^2$ be the  Banach spaces $\R^2$ endowed with the $2$-norm and the sum-norm, respectively. Given $\beta \in (0, 1)$, there exists a bounded linear operator $T_{\beta}: \ell_2^2 \to \ell_1^2$ with $\|T_{\beta}\| = 1$ such that
\begin{itemize}
\item[(i)] $\|T_{\beta}(e_1)\|_1 = \beta$ and
\item[(ii)] for every $z \in S_{\ell_2^2}$ such that $\|T_{\beta} (z)\|_1 = 1$ we have $\|z - e_1\|_1 = \sqrt{2}$.

\end{itemize}

\end{prop}

\begin{proof}
Let $\beta \in (0, 1)$. Define $T_{\beta}: \ell_2^2 \to \ell_1^2$ by
\begin{equation*}
T_{\beta} (x, y) := \left( \frac{ \beta x - y}{2}, \frac{\beta x + y}{2} \right),
\end{equation*}
for all $(x, y) \in \ell_2^2(\R)$. Then
\begin{equation*}
T_{\beta} (e_1) = \left( \frac{\beta}{2}, \frac{\beta}{2} \right) \ \ \mbox{and} \ \ T_{\beta} (e_2) = \left( - \frac{1}{2}, \frac{1}{2} \right).
\end{equation*}
So $\|T_{\beta} (e_1)\|_1 = \beta$ and $\|T_{\beta} (e	_2)\|_1 = 1$. Also, we have $\|T_{\beta}\| = 1$. Indeed, let $(x, y) \in S_{\ell_2^2}$. Then
\begin{equation*}
\|T_{\beta}(x, y)\|_1 = \left| \frac{\beta x - y}{2} \right| + \left| \frac{\beta x + y}{2} \right| = \frac{1}{2} |\beta x - y | + \frac{1}{2} | \beta x + y|.
\end{equation*}
Recall that for every $a, b \in \R$, we may write
$\max \{ a, b \} = \frac{1}{2}(a + b) + \frac{1}{2} |a - b|$. With this in mind, if $|\beta x + y| \geq 0$, then
\begin{equation*}
\|T_{\beta} (x , y) \|_1 = \frac{1}{2} (\beta x + y) + \frac{1}{2} | \beta x - y| = \max \{ \beta x , y \} \leq 1
\end{equation*}
and if $| \beta x + y| \leq 0$ likewise we have
\begin{equation*}
\|T_{\beta} (x, y)\|_1 = \frac{1}{2} [(- \beta x) + (-y)] + \frac{1}{2} |\beta x - y| = \max \{ -\beta x, -y \} \leq 1.
\end{equation*}
Since $\|T_{\beta} (e_2)\|_1 = 1$, we have $\|T_{\beta}\| = 1$. Now, suppose that $z = (a, b) \in S_{\ell_2^2}$ is such that $\|T_{\beta} (z)\|_1 = 1$. So,
\begin{equation*}
1 = \frac{1}{2} |\beta a - b| + \frac{1}{2}|\beta a + b|.
\end{equation*}
Again, we have two cases. If $|\beta a + b| \geq 0$, then $1 = \max \{ \beta a , b \}$ which implies $b = 1$ since $\beta < 1$, and if $|\beta a + b| \leq 0$, then $1 = \max \{ - \beta a, - b \}$ which implies $b = - 1$. Since $a^2 + b^2 = 1$, we obtain $a = 0$ and $b = \pm 1$. This means that if $T_{\beta}$ attains its norm at some $z \in S_{\ell_2^2 }$, then $\|e_1 - z\|_2 = \sqrt{2}$.
\end{proof}

In particular, we get that the pair $(\ell_2^2 (\R), \ell_1^2 (\R))$ fails the uniform sBPBp. Next we show that the pair $(\ell_2^2 (\R), \ell_q^2 (\R))$ for $1 \leq q < 2$ also fails this property.

\begin{prop}\label{ParaMiguel2} Let $\ell_2^2$ and $\ell_q^2$ be the space $\R^2$ endowed with the 2-norm and the $q$-norm, respectively, with $1 \leq q < 2$. Given $\beta \in (0, 1)$, there exists $T_{\beta}: \ell_2^2 \longrightarrow \ell_q^2$ with $\|T_{\beta}\| = 1$ such that
\begin{itemize}
\item[(i)] $\|T_{\beta} (e_1)\|_q = \beta$ and
\item[(ii)] for every $z \in \ell_2^2$ such that $\|T_{\beta} (z)\|_q = 1$ we have $\|z - e_1\|_2 = \sqrt{2}$.
\end{itemize}
\end{prop}

\begin{proof} Let $1 \leq q < 2$ and define $T: \ell_2^2 \longrightarrow \ell_q^2$ by
\begin{equation*}
T(x, y) := \left( \frac{x - y}{2^{\frac{1}{q}}}, \frac{x + y}{2^{\frac{1}{q}}} \right)
\end{equation*}
for every $(x, y) \in \ell_2^2$. First of all, note that
\begin{equation*}
\|T(e_2)\|_q^q = \left| - \frac{1}{2^{\frac{1}{q}}} \right|^q + \left( \frac{1}{2^{\frac{1}{q}}} \right)^q = \frac{1}{2} + \frac{1}{2} = 1,
\end{equation*}	
i.e., $\|T(e_2)\|_q = 1$. This shows that $\|T\| \geq 1$. Next, we show that $\|T\| \leq 1$ and also that the only points which $T$ attains its norm are at $\pm e_1$ and $\pm e_2$. To do so, we study the norm of the operator $T$ by using the following compact set:
\begin{equation*}
K:= \left\{ (a,b)\in \R^2: a^2 + b^2 \leq 1, \ a, b \geq 0 \right\}.
\end{equation*}
By symmetry, the norm of $T$ is the maximum of $\|T(z)\|$ with $z$ in $K$. Let  $z_0 = (a_0, b_0)$ a point of $K$ such that $T$ attains its norm at $z_0$, that is,  $\|T\| = \|T(z_0)\|$. We consider $K_1$ as the segment that connect $(0, 0)$ with $e_1$, $K_3$ as the segment that connect $(0, 0)$ with $e_2$ and $K_2$ as the arc that connect $e_1$ with $e_2$. See Figure \ref{fig:Maximum}.
	
\begin{figure}
	\centering
	\includegraphics[width=0.8\linewidth]{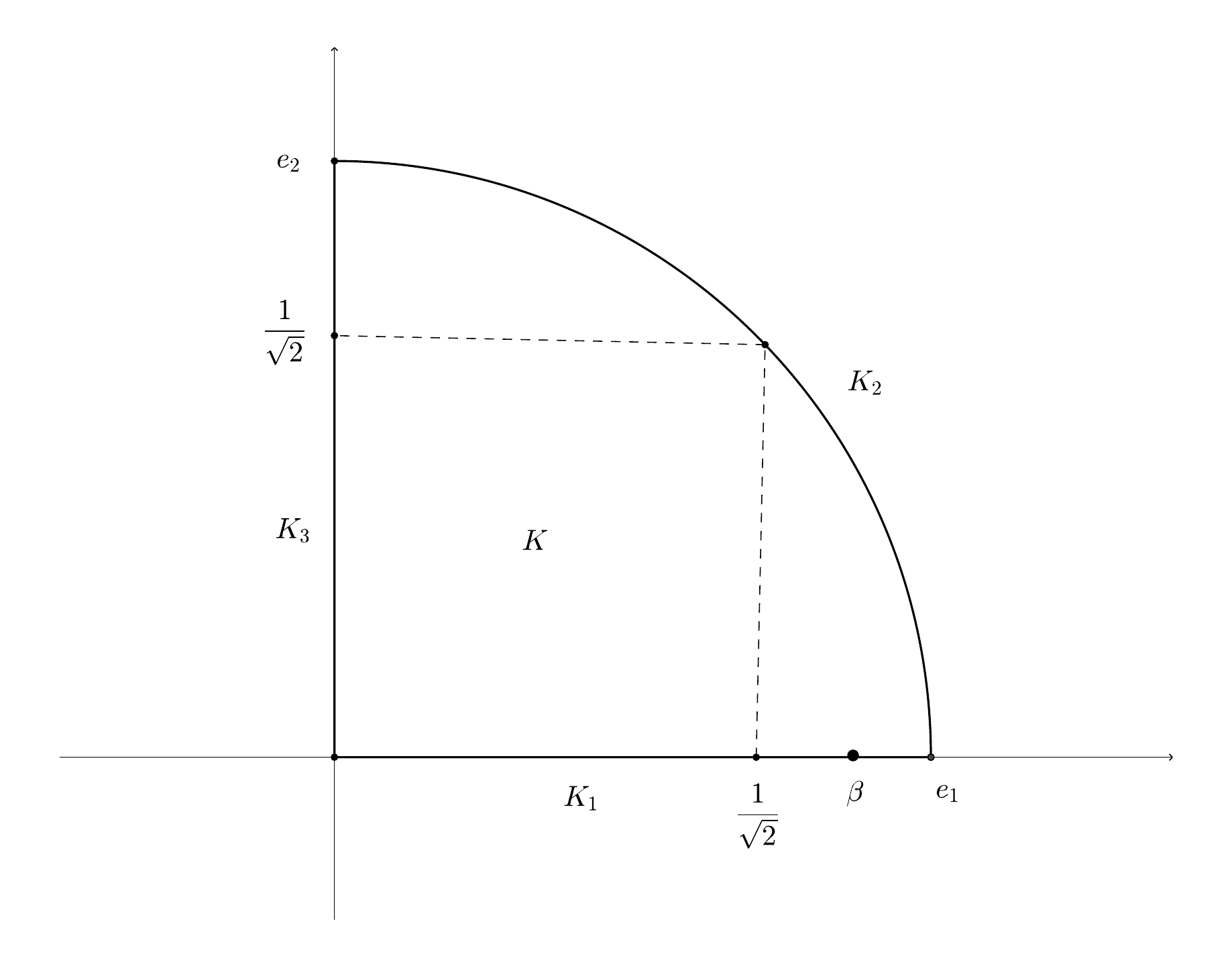}
	\caption{}
	\label{fig:Maximum}
\end{figure}
	
	
It is enough to study the values of $\|T(z)\|_q$ on the set $K_2 \setminus \{ e_1, e_2 \}$ since the operator $T$ attains its norm at elements of the sphere and $\|T(e_1)\|_q = \|T(e_2)\|_q = 1$. We have
$$K_2 \setminus \{ e_1, e_2 \}= \left\{ (x,  f(x)):x \in \left( \frac{1}{\sqrt{2}}, 1 \right) \right\} \cup \left\{ \left( \frac{1}{\sqrt{2}}, \frac{1}{\sqrt{2}} \right) \right\} \cup \left\{ (g(y),y): y \in \left( \frac{1}{\sqrt{2}}, 1 \right) \right\},$$
with $f: \left( \frac{1}{\sqrt{2}}, 1 \right) \to \R$ defined as $f(x)=(1 - x^2)^{\frac{1}{2}}$ and
$g: \left( \frac{1}{\sqrt{2}}, 1 \right) \to \R$ defined as $g(y)=(1 - y^2)^{\frac{1}{2}}$.
Since
\begin{equation*}
\left\| T \left( \frac{1}{\sqrt{2}}, \frac{1}{\sqrt{2}} \right) \right\|_q = \frac{2}{2^{\frac{1}{2} + \frac{1}{q}}} < 1
\end{equation*}
for every  $1 \leq q < 2$, then  $z_0\neq \left( \frac{1}{\sqrt{2}}, \frac{1}{\sqrt{2}} \right)$. (On the other hand observe that if $q = 2$, then $\left\|T(\frac{1}{\sqrt{2}}, \frac{1}{\sqrt{2}}) \right\|_q = 1$ and if $q > 2$, then $\left\|T(\frac{1}{\sqrt{2}}, \frac{1}{\sqrt{2}}) \right\|_q > 1$.) Thus if $z_0\in K_2 \setminus \{ e_1, e_2 \}$, then either $z_0 \in \left\{ (x,  f(x)):x \in \left( \frac{1}{\sqrt{2}}, 1 \right) \right\}$ and then $a_0$ would be a critical point of $F$ in $\left( \frac{1}{\sqrt{2}}, 1  \right)$,
where
$$F(x)=\|T(x,f(x))\|_q^q= \frac{1}{2} \left[ \left( x - (1 - x^2)^{\frac{1}{2}} \right)^q + \left( x + (1 - x^2)^{\frac{1}{2}}) \right)^q \right.
$$
or $z_0 \in \left\{ (g(y),y): y \in \left( \frac{1}{\sqrt{2}}, 1 \right) \right\}$ and in this case $b_0$ would be a critical point of $G$ in $\left( \frac{1}{\sqrt{2}}, 1 \right)$, where
\begin{equation*}
G(y) = \|T(g(y),y)\|_q^q=\frac{1}{2} \left[ \left(  y- (1 - y^2)^{\frac{1}{2}} \right)^q + \left( (1 - y^2)^{\frac{1}{2}} + y \right)^q \right].
\end{equation*}
But, as we will see in the next lines, these can not happen because $F'(x) > 0$  and $G'(y)>0$ for all $x,y \in \left( \frac{1}{\sqrt{2}}, 1 \right)$ and then $z_0\not\in K_2 \setminus \{ e_1, e_2 \}$. Indeed, we consider first the case that $x \in \left( \frac{1}{\sqrt{2}}, 1 \right)$
For every $x \in \left( \frac{1}{\sqrt{2}}, 1 \right)$, we get
\begin{equation*}
F'(x) = \frac{q}{2} \left[ \left( x - (1 - x^2)^{\frac{1}{2}} \right)^{q - 1} \left( 1 + \frac{x}{(1 - x^2)^{\frac{1}{2}}} \right) + \left( x + (1 - x^2)^{\frac{1}{2}} \right)^{q - 1} \left( 1 - \frac{x}{(1 - x^2)^{\frac{1}{2}}} \right) \right].
\end{equation*}
For $x \in \left( \frac{1}{\sqrt{2}}, 1 \right)$, we have that $\left( x - (1 - x^2)^{\frac{1}{2}} \right)^{q - 1} > 0$ and since $$\left( x + (1 - x^2)^{\frac{1}{2}} \right)^{q - 1} \geq \left( x - (1 - x^2)^{\frac{1}{2}} \right)^{q - 1}$$ for every $x$ on this interval, we obtain that
\begin{equation*}
F'(x) \geq \frac{q}{2} \left( x - (1 - x^2)^{\frac{1}{2}} \right)^{q - 1} \left( 1 + \frac{x}{(1 - x^2)^{\frac{1}{2}}} + 1 - \frac{x}{(1 - x^2)^{\frac{1}{2}}} \right) = q \left( x - (1 - x^2)^{\frac{1}{2}} \right)^{q - 1} > 0,
\end{equation*}
for every $x \in (\frac{1}{\sqrt{2}}, 1)$. A simply  change of the  letter $F$ by $G$ and $x$ by $y$ implies that  $G'(y)>0$ for every  $y \in \left( \frac{1}{\sqrt{2}}, 1 \right)$.
	
Everything we did so far was to prove that $\|T\|=1$  and that $T$ attains its norm on $K$ only at $z = e_1$ and $z = e_2$. Therefore, we may conclude that $T$ attains its maximum at $\pm e_1$ and at $\pm e_2$. In other words, we proved that $T(B_{\ell_2^2}) \cap S_{\ell_q^2} = \{ \pm e_1, \pm e_2 \}$.

Now, for $0 < \beta < 1$, define $T_{\beta}: \ell_2^2 \longrightarrow \ell_q^2$ by
\begin{equation*}
T_{\beta} (x, y) = \left( \frac{\beta x - y}{2^{\frac{1}{q}}}, \frac{\beta x + y}{2^{\frac{1}{q}}} \right),
\end{equation*}
for every $(x, y) \in \ell_2^2$. Note that $\|T_{\beta}(e_1)\|_q = \beta$ and $\|T_{\beta} (e_2)\|_q = 1$.  Since $T_{\beta} (B_{\ell_2^2}) \subset T(B_{\ell_2^2}) \subset B_{\ell_q^2}$, then $\|T_{\beta}\| \leq 1$. Also, using that $T(B_{\ell_2^2}) \cap S_{\ell_q^2} = \{ \pm e_1, \pm e_2 \}$ and that $\|T_{\beta} (\pm e_1)\|_q < 1$, we have that $\|T_{\beta}(\pm e_2)\|_q = 1$. This implies that if $z \in S_{\ell_2^2}$ is such that $\|T_{\beta} (z)\|_q = 1$, then $z = \pm e_2$ and therefore $\|e_1 - z\|_2 = \sqrt{2}$ as we wanted.


\end{proof}

As a consequence of Propositions \ref{ParaMiguel1} and \ref{ParaMiguel2} we have the following corollary.
\begin{cor} The pair $(\ell_2^2(\R), \ell_q^2(\R))$ fails the uniform sBPBp for every $1 \leq q \leq \infty$.
\end{cor}

What about the case that $1 < p \leq 2$ and $1 \leq q < 2$? We will study the real case of this right now. Consider $1 < p \leq 2$. Define $Id: \ell_p^2 (\R) \to \ell_2^2 (\R)$ by $Id (x, y) = (x, y)$ for every $(x, y) \in \ell_p^2 (\R)$. Then $Id (e_1) = e_1$ and $Id (e_2) = e_2$. Also, if $(x, y) \in S_{\ell_p^2 (\R)}$, then $|x|^p + |y|^p = 1$ and so $|x|, |y| \leq 1$. This implies that $|x|^2 \leq |x|^p$ and $|y|^2 \leq |y|^p$, and therefore
\begin{equation*}
\|Id (x, y)\|_2^2 = |x|^2 + |y|^2 \leq |x|^p +_|y|^p = 1.
\end{equation*}
Thus $\|Id \| = 1$. Given $0 < \beta < 1$, let $T_{\beta}: \ell_2^2 (\R) \longrightarrow \ell_q^2 (\R)$ be as in the Proposition \ref{ParaMiguel2} with $1 \leq q < 2$. Now, define $\widetilde{T}_{\beta}: \ell_p^2 (\R) \longrightarrow \ell_q^2 (\R)$ by $\widetilde{T}_{\beta} = T_{\beta} \circ Id$. Then $\|\widetilde{T}_{\beta}\| \leq \|T_{\beta}\| \|Id \| = 1$. Also, $\|\widetilde{T}_{\beta} (e_1)\|_q = \|(T_{\beta} \circ Id)(e_1)\|_q = \beta$ and $\| \widetilde{T}_{\beta}(e_2)\|_q = \| (T_{\beta} \circ Id)(e_2)\|_q = \|T_{\beta}(e_2)\|_q = 1$. Suppose that there exists $z \in S_{\ell_p^2 (\R)}$ such that $\|\widetilde{T}_{\beta} (z) \|_q = 1$. Then $\|T_\beta (z)\|_q = 1$ and then, as we can see in the proof of the Proposition \ref{ParaMiguel2}, $z$ must be equals to $e_2$ or $-e_2$. In both cases, we have that $\|e_1 - z\|_p = 2^{\frac{1}{p}}$. We just have proved the following result.

\begin{cor} The pair $(\ell_p^2 (\R), \ell_q^2 (\R))$ fails the uniform sBPBp for $1 < p \leq 2$ and $1 \leq q \leq 2$.
\end{cor}

Next we observe that every time that we put the surpremum norm in the range space, the property fails for any pair of the form $(X, \ell_{\infty}^2)$.

\begin{prop} \label{sBPBp1} The pair $(X, \ell_{\infty}^2)$ fails the uniform sBPBp for all Banach space $X$.
\end{prop}

\begin{proof}
Suppose that there exists $\eta(\e) > 0$ that depends only of a given $\e > 0$ satisfying the property. Let $x_1^*, x_2^* \in S_{X^*}$ and $x_1, x_2 \in S_X$ be such that $x_i^*(x_j) = \delta_{ij}$ for $i, j = 1, 2$. Define $T: X \longrightarrow \ell_{\infty}^2$ by
\begin{equation*}
T(x) := \left( \left(1 - \eta(\e) \right)x_1^*(x), x_2^*(x) \right),
\end{equation*}
for all $x \in X$. Then $\ds \|T(x_1)\|_{\infty} = 1 - \eta(\e)$ and $\|T(x_2)\|_{\infty} = 1$. Moreover, since  $\ds 1 - \eta(\e) < 1$, we have that $\|T\| \leq 1$. This shows that $\|T\| = 1$. Therefore, there exists $z \in S_X$ such that $\|T (z)\|_{\infty} = 1$ and $\|z - x_1\| < \e$. Since $\|T(z)\|_{\infty} = \max \left\{ |(1 - \eta(\e))|x_1^*(z)|, |x_2^*(z)| \right\}$ and $(1 - \eta(\e))|x_1^*(z)| < 1$, we have that $|x_2^*(z)| = 1$. On the other hand, since  $|x_2^*(z - x_1)| \leq \|z - x_1\| < \e$ we get a contradiction, since
\begin{equation*}
1 = |x_2^*(z)| = |x_2^*(z - x_1) + x_2^*(x_1)| = |x_2^*(z - x_1)| < \e < 1.
\end{equation*}
\end{proof}

\begin{example}  It is a consequence of the fact that the pair $(\ell_2^2, \ell_{\infty}^2)$ fails the uniform sBPBp that the pair $(\ell_2^2, C[0, 1])$ also fails this property. Indeed, consider $f_1, f_2 \in C[0, 1]$ positive functions defined on $[0, 1]$ such that $\|f_1\|_{\infty} = \|f_2\|_{\infty} = 1$, $\supp (f_1) \subset \left[0, \frac{1}{2} \right]$ and $\supp (f_2) \subset \left[\frac{1}{2}, 1 \right]$. Define $\varphi: \ell_{\infty}^2 \longrightarrow C[0, 1]$ by
\begin{equation*}
\varphi(x, y)(t) := f_1(t)x + f_2(t)y
\end{equation*}
for all $t \in [0, 1]$ and $(x, y) \in \ell_{\infty}^2$. Then $\varphi$ is linear. Since for $t \in \left[ 0, \frac{1}{2} \right)$ we have $|\varphi(x, y)(t)| = |f_1(t)x|$ and for $t \in \left[\frac{1}{2}, 1 \right]$ we have $|\varphi(x, y)(t)| = |f_2(t)y|$, we get that
\begin{equation*}
\|\varphi(x, y)\|_{\infty} = \sup_{t \in [0, 1]} | \varphi(x, y)(t)| = \max \{ |x|, |y| \} = \|(x, y)\|_{\infty}
\end{equation*}
for all $(x, y) \in \ell_{\infty}^2$, using the fact that $\|f_1\|_{\infty} = \|f_2\|_{\infty} = 1$. This shows that $\varphi$ is a linear isometry between $\ell_{\infty}^2$ and the closed subspace of $C[0, 1]$ generated by $f_1$ and $f_2$. Now suppose by contradiction that the pair $(\ell_2^2, C[0, 1])$ has the uniform sBPBp. Then there exists $\eta(\e) > 0$ for all $\e \in (0, 1)$, the modulus of the uniform sBPBp for this pair. Let $T: \ell_2^2 \longrightarrow \ell_{\infty}^2$ with $\|T\| = 1$ and $z_0 \in S_{\ell_2^2}$ be such that
\begin{equation*}
\|T(z_0)\|_{\infty} > 1 - \eta(\e).
\end{equation*}
Define $S:\ell_2^2 \longrightarrow C[0, 1]$ by $S = \varphi \circ T$. Then for all $(x, y) \in S_{\ell_2^2}$, we have that $T(x, y) \in B_{\ell_{\infty}^2}$ and since $\varphi$ is an isometry, we get that
\begin{equation*}
\|S(x, y)\|_{\infty} = \|\varphi \circ T(x, y)\|_{\infty} = \|T(x, y)\|_{\infty},
\end{equation*}
which implies that $\|S\| = \|T\| = 1$. Also,
\begin{equation*}
\|S(z_0)\|_{\infty} = \|T(z_0)\|_{\infty} > 1 - \eta (\e).
\end{equation*}
So, there exists $z_1 \in S_{\ell_2^2}$ such that $\|S(z_1)\|_{\infty} = 1$ and $\|z_1 - z_0\| < \e$. Since $\|T(z_1)\|_{\infty} = \|S(z_1)\|_{\infty} = 1$, we just have proved that the pair $(\ell_2^2, \ell_{\infty}^2)$ also has the uniform sBPBp which contradicts Example \ref{exsBPBp} and Proposition \ref{sBPBp1}.	
	
\end{example}

Next we show that if $Y$ is a $2$-dimensional Banach space, then the pair $(Y, Y)$ does not have the uniform sBPBp. To do so, we use the existence of the Auerbach base for a finite dimensional Banach space (see, for example, \cite[Proposition 20.21]{Jam}). Let $Y$ be an $n$-dimensional Banach space. Then there are elements $e_1, \ldots, e_n$ of $Y$ and $y_1^*, \ldots, y_n^*$ of $Y^*$ such that $\|e_i\| = \|y_i^*\| = 1$ for all $i$ and $y_i^*(e_j) = \delta_{ij}$ for $i \not= j$. In fact, $\{ e_1, \ldots, e_n \}$ is a base of $Y$ called the  Auerbach base of $Y$.

\begin{prop} Let $Y$ be a $2$-dimensional Banach space. Then the pair $(Y, Y)$ fails the uniform sBPBp.
\label{yy}
\end{prop}

\begin{proof} Let $\{e_1, e_2\}$ and $\{y_1^*, y_2^*\}$ satisfying $\|e_i\| = \|y_i^*\| = 1$ for $i = 1, 2$ and $y_i^*(e_j) = \delta_{ij}$ for $i, j = 1, 2$. Since $\{e _1, e_2 \}$ is a base for $Y$, every $y \in Y$ has an expression in terms of $e_1, e_2, y_1^*$ and $y_2^*$ given by $y = y_1^*(y)e_1 + y_2^*(y) e_2$. Given $\beta \in (0, 1)$, define the continuous linear operator $T_{\beta}: Y \to Y$ by $T_{\beta}(y) = \beta y_1^*(y)e_1 + y_2^*(y)e_2$ for all $y = y_1^*(y)e_1 + y_2^*(y) e_2 \in Y$. Then for all $y \in S_Y$, we have that
\begin{eqnarray}
\|T_{\beta}(y)\| \nonumber &=& \| \beta y_1^*(y)e_1 + y_2^*(y)e_2\| \\ \nonumber
&=& \| \beta y_1^*(y)e_1 + \beta y_2^*(y)e_2 - \beta y_2^*(y)e_2 + y_2^*(y)e_2\| \\ \nonumber
&\leq& \| \beta (y_1^*(y)e_1 + y_2^*(y)e_2) \| + \| (1 - \beta) y_2^*(y)e_2\| \\ \nonumber
&=& \beta \|y\| + (1 - \beta) \|y_2^*(y) e_2\| \\ \nonumber
&\leq& \beta \|y\| + (1 - \beta) |y_2^*(y)| \\ \nonumber
&\leq& \beta + 1 - \beta \\ \nonumber
&=& 1. \nonumber
\end{eqnarray}
Then $\|T_{\beta}\| \leq 1$. Also, note that $\|T_{\beta} (e_2)\| = \|e_2\| = 1$. So $\|T_{\beta}\| = 1$. Now let $y_0 \in S_Y$ be such that $\|T_{\beta}(y_0)\| = 1$. Then, using that
\begin{equation*}
1 = \|T_{\beta} (y_0)\| \leq \beta \|y_0\| + (1 - \beta)|y_2^*(y_0)| \leq 1,
\end{equation*}
we get that $|y_2^*(y_0)| = 1$ and therefore $\|e_1 - y_0\| \geq |y_2^*(e_1) - y_2^*(y_0)| = |-1| = 1$.
Finally, if the pair $(Y, Y)$ has the uniform sBPBp, there exists $\eta(\e) > 0$ satisfying the property. If we put $\beta = 1 - \frac{\eta(\e)}{2}$, there exists an operator $T \in \mathcal{L}(Y, Y)$ such that $\|T\| = 1$, $\|T(e_1)\| > 1 - \eta(\e)$ and for all $y_0 \in S_Y$ which satisfies $\|T(y_0)\| = 1$ is such that $\|e_1 - y_0\| \geq 1$. This is a contradiction and the pair $(Y, Y)$ fails the uniform sBPBp.

\end{proof}

\begin{rem} It is clear but it is worth mentioning that if the pair $(X, Y)$ has the uniform sBPBp, then the pair $(X, Z)$ also has this property for all closed subspace $Z$ of $Y$. This implies, by the above proposition, that the pair $(Y, Y)$ fails the uniform sBPBp for all $n$-dimensional finite space $Y$. Anyway, just for curiosity, if $\dim(Y) = n \geq 2$, the proof of Proposition \ref{yy} works as well in this situation. Indeed, for $\beta \in (0, 1)$ we define $T_{\beta} \in \mathcal{L}(Y, Y)$ by
\begin{equation*}
T_{\beta}(y) = \beta y_1^*(y)e_1 + \beta y_2^*(y)e_2 + \ldots + \beta y_{n-1}^*(y) e_{n-1} + y_n^*(y)e_n
\end{equation*}
for all $y \in Y$, where $\{ e_1, \ldots, e_n\} \subset S_Y$ and $\{ y_1^*, \ldots, y_n^* \} \subset S_{Y^*}$ is given by the Auerbach basis. Then $\|T_{\beta}(e_i)\| = \beta$ for $i \not= n$ and $\|T_{\beta}(e_n)\| = 1$. To prove that $\|T_{\beta}\| \leq 1$, we add and subtract $\beta y_1^*(y)e_1 + \ldots y_{n-2}^*(y) e_{n-2} + \beta y_n^*(y)e_n$ in $\|T_{\beta} (y)\|$ where $y \in S_Y$, to get $\|T_{\beta} (y)\| \leq \beta \|y\| + (1 - \beta)|y_n^*(y)| \leq 1$. Now, it is clear that if $T_{\beta}$ attains its norm at some $y_0 \in S_Y$ then $\|e_i - y_0\| \geq |y_n^*(e_ i - y_0)| = |y_n^*(y_0)| = 1$ for all $i \not= n$.
\end{rem}

\begin{cor} If $Y$ is a Banach space which contains strictly convex $2$-dimensional subspaces, then there exists a uniformly convex Banach space $X$ such that the pair $(X, Y)$ fails the uniform sBPBp.
\end{cor}

\begin{proof}
Indeed, let $Z$ be a subspace of $Y$ such that $Z$ is stricly convex and $\dim (Z) = 2$. Then $X = Z$ is uniformly convex, since $Z$ is finite dimensional. By Proposition \ref{yy}, the pair $(X, Z)$ fails the uniform sBPBp and by the above observation the pair $(X, Y)$ cannot have this property.
\end{proof}

Next we use the negatives results that we got so far about the uniform sBPBp to get examples of pairs $(X, Y)$ that fail the strong Bishop-Phelps-Bollob\'as prooperty. In Remark \ref{sBPBp5} we noted that if $X$ is not reflexive, then the pair $(X, Y)$ does not have the sBPBp. In what follows we get examples of reflexive Banach spaces $X$ such that the pair $(X, Y)$ fails the sBPBp. We also present a complete characterization for the pairs $(\ell_p, \ell_q)$ concerning this property by showing that there are cases that these pairs satisfy the property and other cases not (see Theorem \ref{sBPBp9}). First of all we use the fact, which is showed in the next remark, that the pair $(\ell_2, \ell_2^2)$ fails the uniform sBPBp to get examples of Banach spaces $Y$ such that the pair $(\ell_2, Y)$ fails the sBPBp (see Corollary \ref{sBPBp7}).

\begin{rem} \label{sBPBp4} By Corollary \ref{sBPBp3}, the pair $(\ell_2, Z)$ has the sBPBp if $\dim(Z) < \infty$. But in the case of the uniform sBPBp, we get a negative result. We note that the pair $(\ell_2, \ell_2^2)$ fails the uniform sBPBp. Indeed, suppose by contradiction that this pair satisfies the property. Then given $\e > 0$ there exists $\eta(\e) > 0$ such that whenever $T \in S_{\mathcal{L}(\ell_2, \ell_2^2)}$ and $x_0 \in S_{\ell_2}$ are such that $\|T(x_0)\|_2 > 1 - \eta(\e)$, there is $x_1 \in S_{\ell_2}$ such that $\|T(x_1)\| = 1$ and $\|x_1 - x_0\| < \e$. Since the pair $(\ell_2^2, \ell_2^2)$ fails the uniform sBPBp, there exists some $\e_0 > 0$, a norm one linear operator $R: \ell_2^2 \longrightarrow \ell_2^2$ and a norm one vector $(a_0, b_0) \in S_{\ell_2^2}$ with $\|R(a_0, b_0)\| > 1 - \eta(\e_0)$ such that there is no point $(c_1, c_2) \in S_{\ell_2^2}$ such that $\|R(c_1, c_2)\|_2 = 1$ and $\|(c_1, c_2) - (a_0, b_0)\|_2 < \e_0$. Let $\pi: \ell_2 \longrightarrow \ell_2^2$ be the projection on the first two coordinates, i.e., $\pi ((a_n)_n) := (a_1, a_2)$ for all $(a_n)_n \in \ell_2$. Then $\|\pi\| = 1$. Define $T: \ell_2 \longrightarrow \ell_2^2$ by $T := R \circ \pi$. Then $\|T\| = \|R\| = 1$. Let $x_0 := (a_0, b_0, 0, 0, \ldots) \in S_{\ell_2}$. We have that
\begin{equation*}
\|T(x_0)\| = \|R(a_0, b_0)\| > 1 - \eta(\e_0).
\end{equation*}
Then there exists $x_1 := (c_n)_n \in S_{\ell_2}$ such that $\|T(x_1)\|_2 = 1$ and $\|x_1 - x_0\|_2 < \e_0$. Since
\begin{equation*}
1 = \|T(x_1)\|_2 = \| R(\pi(x_1))\|_2 = \|R(c_1, c_2)\|_2 \leq \|(c_1, c_2)\|_2 \leq \|x_1\|_2 = 1,
\end{equation*}
we get that $\|R(c_1, c_2)\|_2 = \|(c_1, c_2)\|_2 = 1$. On the other hand, $\|(a_0, b_0) - (c_1, c_2)\|_2 \leq \|x_0 - x_1\|_2 < \e_0$. This is a contradiction and then the pair $(\ell_2, \ell_2^2)$ fails the sBPBp as desired.
\end{rem}

The next theorem connects both sBPBp and uniform sBPBp in order to get examples of pairs of Banach spaces $(X, Y)$ that not satisfy the sBPBp by using the examples of the pairs that fail the uniform sBPBp.

\begin{theorem} If the pair $(X, Y)$ fails the uniform sBPBp, then the pair $(\ell_2 (X), \ell_{\infty}(Y))$ fails the sBPBp.
\end{theorem}

\begin{proof} Suppose that the pair $(X, Y)$ fails the uniform sBPBp. Then for each $n \in \N$, there are $\e_0  \in (0, 1)$, $T_n \in S_{\mathcal{L}(X, Y)}$ and $x_n \in S_X$ with
\begin{equation*}
\|T_n(x_n)\| > \frac{n}{n + 1} = 1 - \frac{1}{n+1}
\end{equation*}
such that whenever $x \in S_X$ satisfies $\|x - x_n\| < \e_0$ we have $\|T(x)\| < 1$. Define $T: \ell_2 (X) \longrightarrow \ell_{\infty}(Y)$ by
\begin{equation*}
T((z_n)_n) := (T_n(z_n))_n
\end{equation*}
for every $(z_n)_n \in \ell_2(X)$. Since $\|T_n\| = 1$ for all $n \in \N$, we get that $\|T\| \leq 1$. On the other hand,
\begin{equation*}
\|T\| \geq \|T((x_n)_n)\| = \sup_{n \in \N} \|T_n(x_n)\| \geq \sup_{n \in \N} \frac{n}{n+1} = 1.
\end{equation*}
So $\|T\| = 1$. We suppose that there is $\eta(\e_0, T) > 0$ such that the pair $(\ell_2(X), \ell_{\infty}(Y))$ has the sBPBp with this constant. Denote by $z \overline{e_n}$ the element of $\ell_2(X)$ such that in the $n$-th position is $z$ and in the rest is zero for all $z \in X$. We take $n \in \N$ to be such that
\begin{equation*}
\frac{1}{n+1} < \eta(\e_0, T).
\end{equation*}
Observe that
\begin{equation*}
\|T(x_n \overline{e_n})\| = \|T_n(x_n)\| > 1 - \frac{1}{n + 1} > 1 - \eta(\e_0, T).
\end{equation*}
Thus there is $v = (v_n)_n \in S_{\ell_2(X)}$ such that
\begin{equation*}
\|T(v)\|_{\infty} = 1 \ \ \mbox{and} \ \ \|v - x_n \overline{e_n}\|_2 < \e_0.
\end{equation*}
By the second inequality, we get that $\|v_j\| < \e_0 < 1$ for all $j \not= n$ and $\|v_n - x_n\| < \e_0$. Moreover, by the equality
\begin{equation*}
1 = \|T(v)\|_{\infty} = \sup_{j \in \N} \|T_j (v_j)\|
\end{equation*}
we get that $\|T_n(v_n)\| = 1$ since $\|T_j(v_j)\| < \e_0 < 1$ for all $j \not= n$. So $\|v_n\| = 1$. This contradicts the beggining of the proof and we conclude that the pair $(\ell_2(X), \ell_{\infty}(Y))$ fails the sBPBp whenever $(X, Y)$ fails the uniform sBPBp.
\end{proof}

\begin{cor} There is a infinite dimensional Banach space $Z$ such that the pair $(\ell_2, Z)$ fails the sBPBp.
\end{cor}

\begin{proof} We take $Z = \ell_{\infty} (\ell_2^2)$. Since the pair $(\ell_2^2, \ell_2^2)$ fails the uniform sBPBp, the pair $(\ell_2 (\ell_2^2), \ell_{\infty}(\ell_2^2))$ fails the sBPBp. But $\ell_2 (\ell_2^2)$ is isometric to $\ell_2$. So the pair $(\ell_2, Z)$ fails the sBPBp.
\end{proof}

We give a list of examples of Banach spaces $Y$ such that the pair $(\ell_2, Y)$ fails the sBPBp.

\begin{cor} \label{sBPBp7} Suppose that $(X, Y)$ is one of the following pairs:
\begin{itemize}
\item[(1)] $(\ell_2^2, \ell_{\infty}^2)$,
\item[(2)] $(\ell_2^2, \ell_2^2)$,
\item[(3)] $(\ell_p^2, \ell_q^2)$ for $1 < p \leq q < \infty$.
\item[(4)] $(\ell_2^2, \ell_1^2)$,
\item[(5)] $(\ell_2^2, \ell_q^2)$, for $1 \leq q < 2$.
\item[(6)] $(\ell_2^2 (\R), \ell_q^2(\R))$ for $1 \leq q \leq \infty$.
\item[(7)] $(\ell_p^2(\R), \ell_q^2(\R))$ for $1 < p \leq 2$ and $1 \leq q \leq 2$.
\item[(8)] $(\ell_2^2, C[0, 1])$,
\item[(9)] $(Y, Y)$, where $\dim(Y) = 2$.
\end{itemize}
Then the pair $(\ell_2(X), \ell_{\infty}(Y))$ fails the sBPBp. In particular, the pairs $(\ell_2, \ell_{\infty}(Z))$ fail the sBPBp when
\begin{itemize}
\item[(a)] $Z = \ell_{\infty}^2$, \ $\ell_2^2$, \  $\ell_1^2$, \ $C[0, 1], \ \ell_q^2 \ $ for $2 \leq q < \infty$ in both real and complex cases and
\item[(b)] $Z = \ell_q^2, \ \ell_q^2 \ $ for $1 \leq q \leq 2$ in the real case.
\end{itemize}
\end{cor}

In the next theorem we give a complete characterization of the strong Bishop-Phelps-Bollob\'as property for the pairs $(\ell_p, \ell_q)$.

\begin{theorem} \label{sBPBp9} The following holds.
\begin{itemize}
\item[(i)] The pair $(\ell_p, \ell_q)$ has the sBPBp whenever $1 \leq q < p < \infty$.
\item[(ii)] The pair $(\ell_p, \ell_q)$ fails the sBPBp whenever $1 < p \leq q < \infty$.
\end{itemize}
\end{theorem}

\begin{proof} (i) By the Pitt's theorem every bounded linear operator from $\ell_p$ into $\ell_q$ with $1 \leq q < p < \infty$ is compact. By Theorem \ref{sBPBp2} the pair $(\ell_p, \ell_q)$ has the sBPBp since $\ell_p$ is uniformly convex for $1 < p < \infty$.

(ii) Let $\e \in (0, 1)$ and $1 < p \leq q < \infty$. Consider $\ell_p$ and $\ell_q$ as the Banach spaces $\ell_p(\ell_p^2)$ and $\ell_q(\ell_q^2)$, respectively. For each $n \in \N$ define $T_n: \ell_p^2 \longrightarrow \ell_q^2$ by
\begin{equation*}
T_n(x, y) := \left( \left( 1 - \frac{1}{2n} \right)x, y \right)
\end{equation*}
for each $(x, y) \in \ell_p^2$. Let $z = ((x_n, y_n))_{n \in \N} \in \ell_2$ and let $T: \ell_p \longrightarrow \ell_q$ be defined by
\begin{equation*}
T(z) := (T_n(x_n, y_n))_{n \in \N} = \left( \left(1 - \frac{1}{2n}\right) x_n, y_n \right)_{n \in \N}.
\end{equation*}
For each $z = ((x_n, y_n))_{n \in \N} \in \ell_p$  we have that $\|z\|_p = \left( \sum_{j=1}^{\infty} |x_j|^p + |y_j|^p \right)^{\frac{1}{p}}$ and then
\begin{equation*}
\|T(z)\|_q = \left( \sum_{j=1}^{\infty} \left(1 - \frac{1}{2j}\right)^q |x_j|^q + |y_j|^q \right)^{\frac{1}{q}} \leq \left( \sum_{j=1}^{\infty} |x_j|^q + |y_j|^q \right)^{\frac{1}{q}} = \|z\|_q \leq \|z\|_p.
\end{equation*}
So $\|T\| \leq 1$. We consider the vectors
\begin{equation*}
e_{1, n} := ((0, 0), \ldots, (0, 0), (1, 0), (0, 0), \ldots) \  \mbox{and} \ e_{2, n} := ((0, 0), \ldots, (0, 0), (0, 1), (0, 0), \ldots) \in S_{\ell_p}.
\end{equation*}
Thus we get that $\|T(e_{2, n})\|_q = \|(0, 1)\|_q = 1$. So $\|T\| = 1$. Suppose that there exists $\eta(\e, T) > 0$ such that the pair $(\ell_p, \ell_q)$ has the sBPBp. Let $n \in \N$ be such that $\frac{1}{2n} < \eta(\e, T)$. So since $\|T\| = \|e_{1, n}\|_p = 1$ and
\begin{equation*}
\|T(e_{1, n})\|_q = 1 - \frac{1}{2n} > 1 - \eta(\e, T)
\end{equation*}
there exists $v = (u_n, w_n) \in \ell_p$ such that
\begin{equation*}
\|T(v)\|_q = \|v\|_p = 1 \ \ \ \mbox{and} \ \ \ \|v - e_{1, n}\|_2 < \e.
\end{equation*}
Next we claim that $u_j = 0$ for all $j \in \N$. Indeed, suppose that there exists some $j_0 \in \N$ such that $u_{j_0} \not= 0$. Thus
\begin{eqnarray*}
\|T(v)\|_q &=& \left( \sum_{j=1}^{\infty} \left(1 - \frac{1}{2j} \right)^q |u_j|^q + |w_j|^q \right)^{\frac{1}{q}} \\
&=& \left( \left(1 - \frac{1}{2 j_0} \right)^q |u_{j_0}|^q + |w_{j_0}|^q + \sum_{j \not= j_0} \left(1 - \frac{1}{2j} \right)^q |u_j|^q + |w_j|^q \right)^{\frac{1}{q}} \\
&<& \left( |u_{j_0}|^q + |w_{j_0}|^q + \sum_{j \not= j_0} |u_j|^q + |w_j|^q \right)^{\frac{1}{q}} \\
&=& \|v\|_q \\
&\leq& \|v\|_p = 1
\end{eqnarray*}
which is a contradiction. Then $u_j = 0$ for all $j \in \N$. Because of that, we have
\begin{equation*}
\|e_{1, n} - v\|_q = \left( 1 + \sum_{j \not= n} |w_j|^q \right)^{\frac{1}{q}} \geq 1^{\frac{1}{q}} = 1 > \e.
\end{equation*}
This new contradiction shows that the pair $(\ell_p, \ell_q)$ fails the sBPBp for $1 < p \leq q < \infty$.

\end{proof}

\bigskip

\noindent {\bf Acknowledgement.} This article is based on parts of my Ph.D. thesis at
Universidad de Valencia and I would like to thank Domingo Garc\'ia, Manuel Maestre and Miguel Mart\'in for their wonderful supervision on this work.

\end{document}